\documentclass[10pt]{amsart}

\usepackage{graphicx}
\usepackage{amsmath}
\usepackage{amssymb}
\usepackage{mathrsfs} 
\usepackage{xcolor}

\usepackage[utf8]{inputenc}

\newtheorem{thm}{Theorem}[section]

\newtheorem{lem}[thm]{Lemma}
\newtheorem{cor}[thm]{Corollary}

\newtheorem{qu}[thm]{Question}

\theoremstyle{definition}
\newtheorem{definition}[thm]{Definition}
\newtheorem{example}[thm]{Example}

\theoremstyle{remark}

\numberwithin{equation}{section}

\newcommand{\R}{\mathbb{R}} 
\newcommand{\C}{\mathbb{C}} 
\newcommand{\Z}{\mathbb{Z}}
\newcommand{\G}{\mathbb{G}}
\newcommand{\N}{\mathbb{N}}
\newcommand{\T}{\mathbb{T}}
\newcommand{\Sf}{S^{2n-1}}

\newcommand{\HH}{\mathcal{H}}
\newcommand{\dH}{\dim_{\mathcal{H}}}
\DeclareMathOperator{\M}{M}
\DeclareMathOperator{\dm}{d}
\DeclareMathOperator{\W}{WF}
\DeclareMathOperator{\supp}{supp}

\DeclareMathOperator{\spec}{spec}

\DeclareMathOperator{\im}{im}
\DeclareMathOperator{\Char}{Char}

\newcommand{\mres}{%
  \,\raisebox{-.127ex}{\reflectbox{\rotatebox[origin=br]{-90}{$\lnot$}}}\,%
 }

\begin{document}

\title[Microlocal approach to the Hausdorff dimension of measures]{Microlocal approach to the Hausdorff dimension of measures}

\author{R. Ayoush}
\address{Institute of Mathematics, Polish Academy of Sciences
00-656 Warszawa, Poland}
\email{rayoush@impan.pl}

\author{M. Wojciechowski}
\address{Institute of Mathematics, Polish Academy of Sciences
00-656 Warszawa, Poland}
\email{miwoj.impan@gmail.com}


\begin{abstract}
In this paper we study the dependence of geometric properties of Radon measures, such as Hausdorff dimension and rectifiability of singular sets, on the wavefront set.  We prove our results by adapting the method of Brummelhuis to the non-analytic case. As an application we obtain a general form of uncertainty principle for measures on the complex sphere which subsumes certain classical results about pluriharmonic measures.
\end{abstract}

\maketitle

\section{Introduction}

The purpose of this paper is to extend the programme of \cite{Br} to the case of singular measures in a quantitative way. Namely, in the mentioned paper it is presented how to derive analyticity of measures (in the sense of belonging to the local Hardy-Goldberg space) from the knowledge about their wave front sets (\cite{Br}, Theorem 1.4.). This was obtained by translating properties of Riesz sets to the microlocal setting. Let us recall that a closed set $A\subset{\R^{n}}$ is called a Riesz set, if any $\mu \in \M(\R^{n})$ (the set of finite, Borel regular measures) such that $\supp(\hat{\mu}) \subset A$ is absolutely continuous with respect to the Lebesgue measure. The key point in our modification is the replacement of this notion by the one which allows to control the dimension of singularities:

\begin{definition}
	For any Radon measure~$\mu$ on $\R^{n}$ we define its lower Hausdorff dimension by
	\begin{equation*}
	\dH(\mu) = \inf \{\alpha\colon \hbox{there exists a Borel set~$E$ such that} \ \dH E \leq \alpha \ \hbox{and} \ \mu(E) \neq 0  \}.
	\end{equation*}
\end{definition}

\begin{definition}(\cite{RW})
	We call a closed set $A\subset \R^{n}$ an $s$-Riesz set when $\dH(\mu) \geq s$ for each $\mu \in \M(\R^{n})$ such that $\supp(\hat{\mu}) \subset A$.
\end{definition}

This operation yields the following uncertainty principle which expresses a duality between the size of the wave front set and the Hausdorff dimension:

\begin{definition}
We say that a set $F$ has a $k$--dimensional gap if there exists a $k$--dimensional linear space $V \subset \R^{n}$ with a conic neighbourhood $N_{V}$ such that
\begin{equation*}
	F \cap N_{V} \setminus B(0,r) = \emptyset
\end{equation*}
for some ball $B(0,r)$.
\end{definition}
\begin{thm}\label{microlocal}
	Let $\mu$ be a Radon measure on $\R^{n}$ such that  $\W_{x}(\mu)$ has a $k$--dimensional gap at $\mu$--almost every $x\in \R^{n}$. Then
\begin{equation}\label{UP_ineq}
\dH(\mu) \geq k.
\end{equation}
Moreover, if a $k$--dimensional Borel set $E\subset \R^{n}$ satisfies $\HH^{k}(E) <+\infty$ and \ \ \ $\mu(E) \neq 0$, then there exists a $k$--rectifiable  set $E_{r}\subset E$ such that $|\mu|(E\setminus E_{r})=0$.
\end{thm}
For the definition and basic properties of rectifiable sets we refer the reader to Chapter 18 in \cite{Mattila}. The second part of the theorem asserts about additional regularity of sets minimizing \eqref{UP_ineq}. The above may be thought as a substitute for $k$-Riesz sets in the  non-Euclidean context. In particular, it has several consequences in the study of measures on the complex sphere. Fine properties of such measures were studied, among others,  in \cite{Ale}, \cite{Do}, \cite{Do2}, \cite{Do3}, \cite{For}, \cite{For2}, \cite{HR}, \cite{JR} and \cite{Rog} . 

To state our results, we need to recall some basic notions from the harmonic analysis on $\Sf = \{z \in \C^{n}\colon |z| = 1\}$ (see Chapter 12 in \cite{Rudin} for detailed informations about this topic). In the considerations below  we treat $\Sf$ as a $(2n-1)$--dimensional submanifold of $\R^{2n}$ and the Hausdorff dimension is computed with respect to the Euclidean metric on $\R^{2n}$. As previously, we denote by $\M(\Sf)$ the set of finite, Borel regular measures on $\Sf$. By $\Z_{+}$ we understand the set of non-negative integers, and for $(p,q) \in \Z_{+}^{2}$ the symbol $H(p,q)$ stands for the space of restrictions to $\Sf$ of all harmonic polynomials in $\C^{n}$ which are of degree $p$ in $z_{1},\dots,z_{n}$ and of degree $q$ in $\bar{z_{1}}, \dots, \bar{z_{n}}$. Those spaces form an orthogonal decomposition $L^{2}(\Sf, \sigma) = \bigoplus_{p,q \geq 0} H(p,q)$, where $\sigma$ is the Haar measure on $\Sf$. We call $\pi_{p,q}: L^{2}(\Sf,\sigma) \to H(p,q)$ the orthogonal projection onto $H(p,q)$. This transformation is given by the reproducing kernel $K_{p,q}$ (see \cite{Ale}, p. 118 for the explicit formula) and can be continued to the space of finite measures. This leads to the below definition of spectrum:

\begin{definition}
	For any $\mu \in \M(\Sf)$ we define
	\begin{equation*}
		\spec(\mu) = \Big{\{} (p,q) \in \Z_{+}^{2} \colon \pi_{p,q} \mu (z) = \int_{\Sf} K_{p,q}(z,w) \dm \mu(w) \not\equiv 0 \Big{\}}.
	\end{equation*}  
\end{definition}
Theorem~\ref{microlocal} applies to this setting as follows:
\begin{definition}
	For $E \subset \Sf$ we write $\T \cdot E := \{e^{i t}z \colon z\in E \ \mbox{and} \ t\in[0,2\pi] \}$.
\end{definition}

\begin{thm}\label{UP}
For any $0<\epsilon<1$ let us denote
	\[
	\kappa(\epsilon):= \Big{\{} (x,y) \in (0,+\infty)^{2} \colon 1-\epsilon < \frac{y}{x} < \frac{1}{1-\epsilon} \Big{\}}.
	\]
	 Then, any measure $\mu \in \M(\Sf)$ such that $\spec(\mu) \cap \kappa(\epsilon)$ is finite or empty for some $\epsilon$ satisfies the following regularity property: 
	\begin{equation}\label{strongUP}
		|\mu|(\T \cdot E) = 0 \quad \mbox{if} \ \ \HH^{2n-2}(E) = 0.
	\end{equation}
	Moreover, if  $\mu(E) \neq 0$ for some $(2n-2)$--dimensional Borel set $E\subset \Sf$ such that $\HH^{2n-2}(E) <+\infty$, then there exists a $(2n-2)$--rectifiable  set $E_{r}\subset E$ such that $|\mu|(E\setminus E_{r})=0$.
\end{thm}

\begin{cor}\label{UP2}
	For any $\mu \in \M(\Sf)$ satisfying the assumptions of Theorem \ref{UP} we have
	\begin{equation*}
		\dH(\mu) \geq 2n-2.
	\end{equation*}
\end{cor}

In comparison with Theorem 2.1. in \cite{Br}, the above says that, even after dropping the assumption about strong antisymmetry of spectrum, we can still maintain high regularity under relatively weak Fourier constraints. Natural examples of measures that satisfy them are the so-called pluriharmonic measures. They correspond to the case when the spectrum lies in the sum of horizontal and vertical ray, i.e.
\[
\spec(\mu) \subset \{(p,0) \colon p \in \Z_{+}\} \cup \{(0,q) \colon q \in \Z_{+} \}
\]
 (cf. also Example~\ref{dpluri} in the last chapter). For this case, property~\eqref{strongUP} was proved in full generailty by Aleksandrov (\cite{Ale},  Theorem 3.1.2.), and by Forelli under additional assumption about positivity (\cite{For}, Corollary 1.11.). Let us point out that, by Proposition 3.3.1. in \cite{Ale},  positive pluriharmonic measures must vanish on $(2n-2)$--dimensional $C^{1}$ submanifolds of $\Sf$. This and Theorem~\ref{UP} imply 

\begin{cor}\label{pospluri}
If $\mu$ is a positive pluriharmonic measure, then $\mu(E)=0$ for any $(2n-2)$--dimensional set $E$ such that $\HH^{2n-2}(E) < +\infty$.
\end{cor}

\section{$s$-Riesz and their properties}

In this chapter we list several theorems that we microlocalize in further steps.

\begin{thm}\label{RW} Let $V\subset \R^{n}$ be a $k$--dimensional subspace, $\alpha,\beta \in (0,+\infty)$ and let $S_{V,\alpha,\beta}$ be the complement of the set
\begin{equation*}
\{(\xi_{1},\xi_{2})\in V \times V^{\perp} \colon |\xi_{1}| \geq \alpha, \ |\xi_{2}| \leq  \beta|\xi_{1}|\}.
\end{equation*}
Then $S_{V,\alpha, \beta}$ is a $k$--Riesz set.	
\end{thm}
This is a direct consequence of Theorem 1 in \cite{RW} and is sufficient for proving \eqref{UP_ineq} and Corollary~\ref{UP2}.  However, for our other purposes we need also a slightly stronger form which also follows from the methods applied in \cite{RW}. We enclose its proof for completeness. 

\begin{definition}
Let us call $\Pi_{V}: \R^{n}\rightarrow V$ the orthogonal projection onto $V$.
\end{definition}

\begin{thm}\label{RW_proj}
	Let $V, \alpha, \beta$ be as in Theorem~\ref{RW}. If $\mu \in \M(\R^{n})$ satisfies $\supp(\hat{\mu}) \subset S_{V,\alpha,\beta}$, then
	\begin{equation}\label{proj}
		\HH^{k}_{\mres V}(\Pi_{V}(E)) = 0 \quad \Rightarrow \quad |\mu|(E) = 0.
	\end{equation}
\end{thm}
\begin{proof}
	Denote by $\pi(\mu)= (\Pi_{V})_{*} \mu$ the pushforward of $\mu$ by $\Pi_{V}$, that is the measure satisfying $\pi(\mu)(A) = \mu(A\times V^{\perp})$ for $A \subset V$. For $a \in \R^{n}$ let us define $\tau_{a}\mu$ by the formula $\dm \tau_{a} \mu = e^{-2\pi\langle a, \cdot\rangle} \dm \mu$. For $\xi = (\xi',0),~ t=(t',0) \in V \times V^{\perp}$ we have
	\begin{equation*}
	\pi(\tau_{a}\mu)~\hat{}~(\xi') = \int_{V}e^{-2\pi i \langle \xi', t' \rangle} \dm \pi(\tau_{a}\mu)(t') = \int_{\R^{n}}e^{-2\pi i \langle \xi, s \rangle} \dm \tau_{a}\mu(s) = \hat{\mu}(\xi + a).
	\end{equation*}
	Since $S_{V,\alpha,\beta} \cap (V+a)$ is a bounded set, it is also a Riesz set (folklore, cf. Example 3.7. in \cite{AW}), which implies absolute continuity of $\pi(\tau_{a}\mu)$ with respect to the Lebesgue measure on $V$. In particular, for $E' = \Pi_{V}(E)$ such that $\HH^{k}_{\mres V}(E')=0$ we get
	\begin{equation*}
		\tau_{a}\mu(E'\times V^{\perp}) = \pi(\tau_{a}\mu)(E') = 0.
	\end{equation*}
  Thus, for any $a\in \R^{n}$
  \begin{equation*}
  \mu_{\mres E'\times V^{\perp}}~\hat{}~(a) =\int_{E' \times V^{\perp}}e^{-2\pi i \langle a,s \rangle} d\mu(s) = \tau_{a} \mu(E'\times V^{\perp}) = 0,
  \end{equation*}
  and finally $\mu_{\mres E' \times V^{\perp}} \equiv 0$ by the uniqueness theorem.
\end{proof}

With a little help of the Besicovitch-Federer projection theorem (\cite{Mattila}, Theorem 18.1) we can adjust the above for dealing with rectifiability of singular sets.

\begin{thm}\label{BesFEd}
Let $\mu \in \M(\R^{n})$ be as in the previous theorem. Then, for any $k$--dimensional Borel set $E$ such that $\HH^{k}(E) < +\infty$ and $\mu(E) \neq 0$ there exists a $k$--rectifiable set $E_{r} \subset E$ such that $|\mu|(E \setminus E_{r}) = 0$. 
\end{thm}

\begin{proof}
	Let us begin with an observation that, for $k$--dimensional vector spaces \\ $W \in \G(k,\R^{n})$, satisfying the formula
	\begin{equation}
		\forall a\in\R^{n} \quad S_{V, \alpha, \beta} \cap (W+a) \ \ \mbox{is a bounded set}
	\end{equation}
	is an open condition in the natural topology on the Grassmannian $\G(k,\R^{n})$. Thus, the same proof as in Theorem~\ref{RW_proj} gives even stronger statement: There exist $\mathcal{O}_{V} \subset \G(k,\R^{n})$, a neighbourhood of $V$ of positive Haar measure, such that 
	\begin{equation}\label{rect}
		\HH^{k}_{\mres W}(\Pi_{W}(F)) = 0 \quad \Rightarrow \quad |\mu|(F) = 0
	\end{equation}
 	is true for $W \in \mathcal{O}_{V}$ and any $\mu$ satisfying $\supp(\mu) \subset S_{V, \alpha, \beta}$.
 	
 	By Theorem 18.1. in \cite{Mattila} we can decompose $E = E_{r} \cup E_{u}$ into disjoint sum of $k$--rectifiable and $k$--unrectifiable set. It suffies to apply the Besicovitch-Federer projection theorem and \eqref{rect} to obtain $|\mu|(E_{u}) = 0$.
\end{proof}

\section{Proofs}

For the convenience of the reader we begin with recalling some basic facts about the wave front set (see Chapter 8 in \cite{Hormander1} or \cite{Taylor}). 

If $\nu \in \mathcal{E}'(\R^{n})$, then we define $\Sigma(\nu)$ as the set of those $\xi \in \R^{n}\setminus \{0\}$, for which there is no conic neighbourhood $C$ such that
\begin{equation}\label{property1}
	\forall_{N \in \N} \exists_{C_{N}>0} \ \ |\hat{\nu}(\xi')| \leq C_{N}(1+|\xi'|)^{-N} \quad \mbox{for}\ \xi' \in C.
\end{equation}
For an arbitrary $\nu \in \mathcal{D}'(\R^{n})$ we define
\begin{equation}\label{property2}
	\W_{x}(\nu) = \bigcap_{\phi} \big{\{} \Sigma(\phi \nu) \colon \ \phi \in C^{\infty}_{c}(\R^{n}), \ \phi(x) \neq 0 \big{\}}.
\end{equation}
For any set $V\subset \R^{n}$ which is a conic neighbourhood of $\W_{x}(\nu)$, there exists a neighbourhood of $x$, say $U_{x}$ such that
\begin{equation}\label{property3}
	\W_{x}(\nu) \subset \Sigma(\phi \nu) \subset V \quad \mbox{for any} \ \phi \in C^{\infty}_{c}(U_{x}), \ \phi(x) \neq 0.
\end{equation}
If $\Phi \colon M \rightarrow N$ is a $C^{\infty}$--diffeomorphism between manifolds $M$ and $N$,  then the wave front set of the pullback $\Phi^{*}\nu$ is described by
\begin{equation}\label{pullback}
\W_{x}(\Phi^{*}\nu) = \{D\Phi^{t}(x)\eta \colon \eta \in \W_{\Phi(x)}(\nu)\}
\end{equation}
For any distribution $\nu \in \mathcal{D}'(M)$ defined on a manifold $M$, we set 
\[
\W(\nu) = \{(x,\xi) \in T^{*}M\setminus\{0\}: \xi \in \W_{x}(\nu)\}.
\]

\begin{proof}(of Theorem~\ref{microlocal}) Let us begin with proving the dimension bound. We may assume that our measure has a $k$--dimensional gap at every $x$. By the property~\eqref{property3} and the assumptions we have that for each $x \in \R^{n}$ there exists a neighbourhood $U_{x}$ such that $\Sigma(\phi\mu) \subset S_{V(x), \alpha(x), \beta(x)}$ for some $k$--dimensional space $V(x)$, some $\alpha(x), \beta(x) \in (0,+\infty)$, and any $\phi \in C^{\infty}_{c}(U_{x})$ such that $\phi(x) \neq 0$. Let us fix $\phi$.  After slightly changing $\alpha(x)$ and $\beta(x)$ if needed, we can construct a function $\eta \in C^{\infty}(\R^{n})$ such that
	\[ \eta(\xi)=
	\begin{cases}
	 0 \quad \mbox{on} \ \Sigma(\phi \mu), \\
     1  \quad \mbox{on} \ \R^{n}\setminus S_{V(x),\alpha(x), \beta(x)}.
	\end{cases}
	\]
From this and \eqref{property1} we obtain that $f(x) = (\eta \cdot \widehat{\phi \mu})~\check{} \in \mathcal{S}(\R^{n})$ and the measure $\phi \mu - f dx$ satisfies the assumptions of Theorem~\ref{RW}. Since modifications of measures by absolutely continuous ones do not have any influence on singular sets, we get that
\begin{equation}\label{localization}
\forall_{x} \exists \mbox{ a neighbourhood} \ U_{x} \ \mbox{s.t.}\ \dH(\phi \mu) \geq  k \quad \mbox{for} \ \phi \in C^{\infty}_{c}(U_{x}), \ \phi(x) \neq 0.
\end{equation}

Suppose by contradiction that there exists $F$ such that $\dH F < k$ and $\mu(F) \neq 0$. By the regularity of $\mu$, we may assume that $F$ is compact, which provides the existence of  a finite cover $F \subset \cup_{j}^{N} U_{x_{j}}$ with sets $U_{x_{j}}$ satisfying \eqref{localization}. Let $\{\phi_{j}\}_{j=1}^{N}$ be a smooth partition of unity inscribed in $\{U_{x_{j}}\}_{j=1}^{N}$. We have
\[
\mu(F) = \sum_{j=1}^{N} \phi_{j}\mu(F) = 0,
\]
which gives the first part of the theorem. To get the rectifiability part we simply replace the use of Theorem~\ref{RW} by Theorem~\ref{BesFEd} in the reasoning above.
\end{proof}

Before proving Theorem~\ref{strongUP} let us remark that, since diffeomorphisms are locally bi-Lipschitz, we can obtain full information about dimension and rectifiability of $\mu$ from the knowledge about pushforward measures $\Phi_{*}\mu$, provided that we have sufficiently many good maps $\Phi \colon \Sf \rightarrow \R^{2n-1}$. Having in mind Theorem~\ref{RW_proj} and the fact that the action of $\T$ defines a foliation of $\Sf$, we  can construct maps tailored to the proof of Theorem~\ref{UP}. 

Recall (\cite{Ale}, Subsection 1.4.) that the (real) tangent space $T_{z}\Sf$ can be decomposed into an orthogonal sum $T_{z}^{\C}\Sf \oplus \R i z$, where 
\[T_{z}^{\C}\Sf= \{\xi \in \C^{n} \colon \langle \xi,z \rangle_{\C^{n}} =0 \}. \]
We introduce coordinates $(\xi_{1},\xi_{2})\in T_{z}^{\C} \Sf \times \R \cong  T_{z} \Sf$ accordingly to this splitting: $\xi = \xi_{1} + \xi_{2}iz$.

\begin{lem}\label{technical}
Suppose that $\Phi \colon \Sf \rightarrow \R^{n}$ is a smooth diffeomorphism and $\mu \in \M(\Sf)$. Let $\nu_{1}=\Phi_{*}\mu$ (understood as a pushforward measure) and $\nu_{2} = (\Phi^{-1})^{*}\mu$ (understood as a pullback of a distribution). Then $\nu_{1}$ and $\nu_{2}$ are mutually absolutely continuous and $\W(\nu_{1}) = \W(\nu_{2})$.
\end{lem}
\begin{proof}
It follows from the formula $\dm \nu_{2} = |\det D\Phi| \dm \nu_{1}$.
\end{proof}
\begin{lem}\label{maps}
Suppose that $E\subset\Sf$ satisfies $\HH^{2n-2}(E) = 0$ . Then, for any  $\mu \in \M(\Sf)$ and $z_{0} \in \Sf$ there exists an open neighbourhood $U_{z_{0}} \subset \Sf$ and a smooth diffeomorphism $\Phi \colon U_{z_{0}} \to V \subset T_{z_{0}}\Sf$, such that 
\begin{enumerate}
	\item $\Phi(z_{0}) = 0$ and $\W_{0}(\Phi_{*}\mu) = \W_{z_{0}}(\mu)$,
	\item $\HH^{2n-2}(\Pi_{V} \Phi(\T \cdot E \cap U_{z_{0}})) = 0$ for $V=T^{\C}_{z_{0}} \Sf$.
\end{enumerate}
\end{lem}
\begin{proof}
Let $\psi: \Sf\cap B(z_{0},\epsilon) \rightarrow T_{z_{0}}\Sf$ be the orthogonal projection onto $T_{z_{0}}\Sf$. Here we choose small $\epsilon$ so that $\psi$ was a bi-Lipschitz diffeomorphism. As we have already mentioned, since $\T$ acts on $\Sf$ (freely) by multiplication, $\Sf$ is foliated by leaves of the form $\{e^{it}\xi\}_{t \in [0,2\pi]}$ (see Theorem 11.3.9. in \cite{Candel}). Thus, $\im \psi$ is foliated by leaves $\{\psi(e^{it}\xi)\}$.

Let us define a function $\gamma$ on $\im \psi$ so that $p=(\xi_{1},\xi_{2}) \in \im \psi \subset T_{z_{0}} \Sf$ is mapped to a point $(\tilde{\xi_{1}},\xi_{2})$, where $(\tilde{\xi_{1}}, 0)$ is the intersection point of the leaf $\{\psi(e^{it}\xi)\}$ containing $p$ with $T_{z_{0}}^{\C}\Sf$. If $\epsilon$ is small enough, then $\gamma$ is well defined as the leaves of foliation are transversal to $T_{z_{0}}^{\C}\Sf$ near $0$. Moreover, $\gamma$ is a diffeomorphism and $\Phi = \gamma \circ \psi$ is the desired map. Indeed, point $(2)$ is satisfied because $\Phi$ and $\Pi_{V}$ are Lipschitz. To prove $(1)$, let us observe that $D\psi(z_{0}) = Id$ (since $D\psi^{-1}(0) = Id$) and $D\gamma(0) = Id$, as we have $\frac{d}{d\xi_{2}}\gamma(0) = (0,1)$ and $\gamma_{| T^{\C}_{z_{0}}\Sf} = Id_{|T^{\C}_{z_{0}}\Sf}$. It remains to use the previous lemma and \eqref{pullback}.
\end{proof}

\begin{proof} (of Theorem~\ref{UP})
	We essentially follow the steps of the proof of Theorem 2.1. in \cite{Br}. Our aim is to show that at each point $z$, $\W_{z}(\mu)$ has a $(2n-2)$--dimensional gap given by $T^{\C}_{z}\Sf$.
	Let us identify $T\Sf$ with $T^{*}\Sf$ using the Euclidean metric on $\R^{2n}$. Take two commuting, first order pseudodifferential operators
	\[
	T_{1} f(z) = \frac{1}{i} \frac{d}{dt} f(e^{it}z) \Big{|}_{t=0}
	\]
	and
	\[
	T_{2} = \sqrt{-\Delta_{\Sf}+(n-1)^{2}Id}-(n-1)Id.
	\]
	The symbol $\Delta_{\Sf}$ stands for the spherical Laplacian. The eigenspaces of $\Delta_{\Sf}$ are $\mathcal{H}(j) = \bigoplus_{p+q=j} H(p,q)$ and the corresponding eigenvalues $\lambda_{j} = -j(j+2n-2)$.  Principal symbols  of $T_{1}$ and $T_{2}$ are
	\[
	\sigma(T_{1})(z,\xi_{1},\xi_{2}) = \xi_{2}
	\]
    and
	\[
	\sigma(T_{2})(z,\xi_{1},\xi_{2}) = c\sqrt{\lvert \xi_{1}\rvert^{2}+\lvert \xi_{2}\rvert^{2}}
	\]
	for some constant $c$. The space $H(p,q)$ can be described as the common eigenspace of $T_{1}$ and $T_{2}$ with eigenvalues $p-q$ and $p+q$, respectively. 
	Let $\chi \in C^{\infty}(\R^{2})$ be any function having the following properties:
	\begin{enumerate}
	\item $\chi(x,y)$ is $0$--homogeneous on $\kappa(\epsilon)$,
	\item $\chi(x,y)\neq0$ on $\kappa(\epsilon)$,
	\item $\chi(x,y) = 0$ outside $\kappa(\epsilon)\cup B(0,\delta)$, for some small $\delta$.
	\end{enumerate}
	By the functional calculus from \cite{stri} (cf. also \cite{Taylor}, Chapter 12), the operator 
	\[
	T_{3} = \chi\Big{(} \frac{T_{1}+T_{2}}{2}, \frac{T_{2}-T_{1}}{2} \Big{)},
	\]
	defined by the spectral theorem, is equal to a $0$--th order pseudo-differential operator with principal symbol
	\begin{equation}\label{principal}
		\sigma(T_{3})(z,\xi_{1},\xi_{2}) = \chi\Big{(} \frac{c\sqrt{\lvert \xi_{1}\rvert^{2}+\lvert \xi_{2}\rvert^{2}} + \xi_{2}}{2}, \frac{c\sqrt{\lvert \xi_{1}\rvert^{2}+\lvert \xi_{2}\rvert^{2}}-\xi_{2}}{2} \Big{)}.
	\end{equation}
	Moreover, by the assumptions we have
	\begin{equation}\label{Tfree}
	T_{3}\mu = \sum_{p,q \geq 0} \chi( p,q)\pi_{p,q}\mu \in C^{\infty}(\Sf).
	\end{equation}
	Theorem 18.1.28 from \cite{Hormander3} says that
	\begin{equation}
 	\W(\mu) \subset \W(T_{3}\mu) \cup \Char(T_{3}).
	\end{equation}
	From this and \eqref{Tfree} we obtain that $\W(\mu)$ is contained in $\Char(T_{3})$, i.e. the set of $(z,\xi) \in T^{*}\Sf \setminus \{0\}$ such that
	\begin{equation}
	\sigma(T_{3})(z,\xi_{1},\xi_{2}) = 0
	\end{equation}
	(we used homogenity of the symbol). It suffices to show that 
	\[
	\Char(T_{3})(z,\cdot) \subset S_{T_{z}^{\C}(\Sf),\alpha, \beta}
	\]
	for parameters $\alpha, \beta$ depending on $\epsilon$ only. Suppose by contradiction that there exists a sequence $\xi^{(n)} = \xi^{(n)}_{1}+iz\xi^{(n)}_{2} \in \Char(T_{3})(z,\cdot)$ such that $|\xi^{(n)}_{2}| \leq \alpha_{n}|\xi^{(n)}_{1}|$ with $\alpha_{n} \downarrow 0$. This means that
	\[
	\frac{y_{n}}{x_{n}} := \frac{c\sqrt{\lvert \xi_{1}^{(n)}\rvert^{2}+\lvert \xi_{2}^{(n)}\rvert^{2}} - \xi_{2}^{(n)}}{c\sqrt{\lvert \xi_{1}^{(n)}\rvert^{2}+\lvert \xi_{2}^{(n)}\rvert^{2}} + \xi_{2}^{(n)}} \to 1
	\]
	as $n\to +\infty$ and  $\chi (\frac{x_{n}}{2},\frac{ y_{n}}{2}) = 0$. This cannot hold by the definition of $\chi$.
	
	To finish the proof, it remains to use Lemma~\ref{maps} in combination with Theorem~\ref{RW_proj}, Theorem~\ref{BesFEd} and an argument analogous to the covering argument from the proof of Theorem~\ref{microlocal}.
 \end{proof}

\begin{proof}(of Corollary~\ref{pospluri})
  By Theorem~\ref{UP}, we may assume that $E$ is rectifiable. Whitney extension theorem (\cite{fed}, Theorem 3.1.16.) says that $E$ is contained in a countable union of $(2n-2)$--dimensional $C^{1}$--submanifolds, thus in view of Proposition 3.3.1. from \cite{Ale} we have $\mu(E) = 0$.
\end{proof}

\section{Examples}

\begin{example}
Let $\mu \in \M(\R^{n})$ be the uniform Hausdorff measure on a smooth $k$--dimensional manifold $M$. Then, at each $x\in M$, $\W_{x}(\mu)$ has a $k$-dimensional gap given by the tangent space.
\end{example}

\begin{example}(cf. \cite{Br2}, p. 140)
For any $\xi \in \Sf$, the one-dimensional Hausdorff measure on the set $\T \cdot \{\xi\}$ has its spectrum equal to $\{(p,p) \colon p\in \Z_{+}\}$. Thus, in Theorem~\ref{UP}, the semi-axis $\{(x,x): x \in (0,+\infty)\}$ cannot be replaced by any other one.
\end{example}

\begin{example}\label{dpluri}
Our results can be applied to the study of the so-called $d$--pluriharmonic measures, introduced in \cite{Do}. They are those measures from $\M(\Sf)$, whose spectrum lies inside
\[
 \{(p,q) \in \Z_{+}^{2}\colon (p-d)(q-d)=0 \ \mbox{and} \ \ p,q\geq d\}.
\]
In particular, the above class contains the classical pluriharmonic measures as $0$--pluriharmonic measures.
\end{example}

We would like to finish the article by leaving the following open problem:

\begin{qu}
Is the dimension bound from Corollary~\ref{UP2} sharp?
\end{qu}

\bibliography{biblio} 
\bibliographystyle{alpha}

\end{document}